\documentclass{amsart}
\usepackage{amsfonts}
\usepackage{amsmath}
\usepackage{amssymb}
\usepackage{graphicx} 
\usepackage{version}

\theoremstyle{plain}
\newtheorem{theorem}{Theorem}[section]
\newtheorem{cor}[theorem]{Corollary}
\newtheorem{lemma}[theorem]{Lemma}
\newtheorem{prop}[theorem]{Proposition}
\theoremstyle{definition}
\newtheorem{definition}[theorem]{Definition}

\newtheorem{question}{Question}
\newtheorem{conjecture}[question]{Conjecture}
\numberwithin{equation}{section}

\begin{document}

\title{A Lower Bound on the Average Size of a Connected Vertex Set of a Graph}
\author[A. Vince]{Andrew Vince}
\address{Department of Mathematics \\ University of Florida \\ Gainesville, FL 32611, USA}
\email{\tt  avince@ufl.edu} 
\thanks{This work was partially supported by a grant from the Simons Foundation (322515 to Andrew Vince).}

\maketitle

\begin{abstract}  
  The topic is the average order of a connected induced subgraph of a graph.  This generalizes, to graphs in general,
 the average order of a subtree of a tree.  In 1984, Jamison proved that the average order, over all trees of order $n$, is minimized by the path $P_n$. In 2018, Kroeker, Mol, and Oellermann conjectured that $P_n$ minimizes the average order over all connected graphs.  
The main result of this paper confirms this conjecture.
\end{abstract}

\section{Introduction} \label{sec:intro}

 Although connectivity is a basic concept in graph theory, problems involving 
the enumeration of the connected induced subgraphs 
of a graph have only recently received attention.  The topic of this paper is the average order of a connected induced subgraph 
of a graph.  Let $G$ be a connected finite simple graph with vertex set $V$, and let $U\subseteq V$. 
 The set $U$ is said to 
be a {\bf connected set} if the subgraph of $G$ induced by $U$ is connected.  Denote the collection of all connected sets, excluding the emptyset, by $\mathcal C = \mathcal C(G)$.  The number of connected sets in $G$ will be denoted by $N(G)$.   Let
\[S(G) = \sum_{U \in \mathcal C} |U|  \] 
be the sum of the sizes of the connected sets.
Further, let
\[ A(G) = \frac{S(G)}{N(G)}   \qquad  \qquad \text{and}  \qquad  \qquad   D(G) =\frac{ A(G)}{n} \]
denote, respectively, the average size of a connected set of $G$ and the proportion of vertices in an average size connected set.  The parameter $D(G)$ is referred to as the {\bf density} of connected sets of vertices.    The density
allows us to compare the  average size of connected sets of graphs of different orders.  The density is also the probability that a vertex chosen at
random from $G$ will belong to a randomly chosen connected set of $G$.  If, for example, $G$ is the complete graph $K_n$, then $A(K_n)$ is the average size of a subset of an $n$-element set, which is $n/2$ (counting the empty set for simplicity), the density then being $1/2$.

There are a number of papers on the average size and density of connected  sets in trees.  The
invariant $A(G)$, in this case, is the average order of a subtree of a tree. Although results are known for trees, beginning with Jamison's 1984 paper \cite{J}, nearly nothing is known for graphs in general.    We review the literature in  Section~\ref{sec:lit}.  Concerning lower bounds, Jamison proved that the density, over all trees of order $n$, is minimized by the path $P_n$.  In particular $A(T) \geq (n+2)/3$ for all trees $T$ with equality only for $P_n$; therefore $D(T) > 1/3$ for all trees.  
Kroeker, Mol, and Oellermann conjectured in their 2018 paper \cite{O} 
that $P_n$ minimizes the average size of a connected set over all connected graphs.  The main result of this paper confirms this conjecture.  

\begin{theorem} \label{thm:main}  If $G$ is a connected graph of order $n$, then 
\[A(G) \geq \frac{n+2}{3},\]
with equality if and only if $G$ is a path.  In particular, $D(G) > 1/3$ for all connected graphs $G$.  
\end{theorem}

After reviewing the relevant literture in Section~\ref{sec:lit}, each of the Sections~\ref{sec:vertex}, \ref{sec:nt}, \ref{sec:x} and \ref{sec:cut} 
contain a  preliminary result required for the proof of Theorem~\ref{thm:main}.  In Section \ref{sec:x}, the result (Theorem~\ref{thm:m1}) concerns the average size a connected set of $G$ containing a fixed connected connected subset $H$.  In Section~\ref{sec:nt}, 
the result (Corollary~\ref{cor:nt}) is that certain very sparse  graphs satisfy the inequality in Theorem~\ref{thm:main}.   
In Section \ref{sec:x}, the result (Theorem~\ref{thm:av}) gives an inequality relating  the
number of connected sets containing a given vertex $x$ to the number of connected sets not containing  $x$. In  Section \ref{sec:cut}, the result
(Theorem~\ref{thm:inequal}) is an essential inequality valid for graphs with at least one cut-vertex.  Section~\ref{sec:lb} provides the final step
in the proof of Theorem~\ref{thm:main}.  Two problems that remain open are discussed in Section~\ref{sec:open}.

\section{Previous results} \label{sec:lit}

Following Jamison's study \cite{J},  a number of papers on the average order of a subtree of a 
tree followed \cite{H,J2,MO,SO,V,WW,YY}.  Concerning upper bounds, Jamison \cite{J}
 provided a sequence of trees (certain ``batons") showing that there are trees with density arbitrarily close to $1$.  
However, if the density $D(T_n)$  of a sequence  $T_n$ of trees tends to $1$, then the proportion of vertices of 
degree $2$ must also tend to $1$.  This led  to the question of upper and lower bounds on the density for trees whose internal
vertices have degree at least three.  Vince and Wang \cite{V} proved that if $T$ is a tree all of whose internal vertices have degree at least three, 
then $\frac 12 \leq D(T) < \frac34.$  Both bounds are best possible in the sense that there exists an
infinite sequence  $\{ ST_n\}$ of trees (stars, for example) such that $\lim_{n \rightarrow
\infty} D(ST_n) = 1/2$ and an infinite sequence $\{ CAT_n\}$ of trees (certain ``caterpillers") such that
$\lim_{n \rightarrow \infty} D(CAT_n)= 3/4$.  

A subtree of a tree $T$ is a connected induced subgraph of $T$.  So it is natural to extend from trees to graphs  $G$ by
asking about the average order of a connected induced subgraph of $G$ - or, in our terminology, the average
size of a connected set of vertices of $G$.   Kroeker, Mol, and Oellermann \cite{O} carried out such an 
investigation for cographs, i.e., graphs that contain no induced $P_4$.   For a connected cograph 
$G$ of order $n$, they proved that 
$n/2 < A(G) \leq (n+1)/2$, with equality on the right if and only if $n = 1$.  Complete bipartite graphs 
are examples of cographs.  In fact, cographs have the following known characterization: a  graph  $G$ is
a cograph if and only if $G = K_1$ or there exist two cographs $G_1$ and $G_2$ such that either $G$ is the disjoint
union of $G_1$ and $G_2$ or $G$ is obtained from the disjoint union by adding all edges joining the vertices
of $G_1$ and $G_2$.  Proving bounds on $A(G)$ for cographs is therefore amenable to an inductive approach 
not applicable to graphs in general.  Balodis, Mol, and Oellermann \cite{B}  proved that for block graphs of order $n$,  i.e., graphs for
which each maximal $2$-connected component is a complete graph, the path $P_n$ minimizes the average size of connected set.
A tree is a block graph, thus their result extends Jamison's lower bound from trees to block graphs.  Theorem~\ref{thm:main}
extends this lower bound to all connected graphs.  

\section{The Average Size of Connected Sets Containing a Given Connected Set} \label{sec:vertex}

If $V$ is the set of vertices of a connected graph $G$ and $H$ is a connected subset of $V$, let $N(G,H), S(G,H)$, and $A(G,H)$ denote the number of connected sets in $G$ containing $H$, the sum of the sizes of all connected sets containing $H$, and the average size of a connected set containing $H$, respectively.  

\begin{theorem} \label{thm:m1} If $H \subseteq V$ is a connected subset of size $h\geq 1$ of a connected graph $G$ of order $n$, then
\[A(G,H) \geq \frac{n+h}{2}.\]
\end{theorem}

\begin{proof}  The proof is by induction on the integer $d = n-h$.  If $d= 0$, then $H = V, \, N(G,H) = 1$ and $S(G,H) = n$.  Therefore $S(G,H) = n = 
 \frac{n+n}{2} \cdot 1 = \frac{n+h}{2} \, N(G,H)$.  This is the base case of the induction. Assume that the statement is true for $d-1$ and let $(G,H)$ be such that
$n-h = d$.  The remainder of the proof is divided into two cases.  Let $Q$ be the set of vertices that are adjacent to some vertex in $H$ but are not in $H$.  
\vskip 1mm

{\it Case 1.}  Assume that there is a vertex $x$ in $Q$ that is not a cut-vertex of $G$.  Let  $G' = G \setminus \{x\}$, which is a connected graph. For simplicity
we use the notation $H+x = H\cup\{x\}$. By the induction hypothesis  
\[ \begin{aligned} S(G,H) &= S(G',H)  + S(G,H+x) \geq \frac{(n-1)+h}{2} N(G',H)+ \frac{n+(h+1)}{2} N(G, H+x) \\ 
&= \frac{n+h}{2} (N(G',H) + N(G, H+x) ) + \frac{1}{2} (N(G,H+x) - N(G', H) ) \\ &= \frac{n+h}{2} N(G,H) + \frac{1}{2} (N(G,H+x) - N(G', H) ) 
\geq \frac{n+h}{2} N(G,H).\end{aligned}\]
The last inequality follows because, for each connected set $U$ counted in $N(G',H)$, the connected set $U+x$ is counted in $N(G,H+x)$.  (Note that
this is not true if $h=0$.)
\vskip 1mm

{\it Case 2.}  Assume that all vertices in $Q$ are cut-vertices of $G$, and let $x$ be one of these vertices.  Let $G' = G-x$.  Let $G_1$ be the connected component of $G-x$ containing $H$ and let $G_2$ be the union of the other components.  Denote the vertex set of $G_2$ by $V_2$, and let $m = |V_2|$.   Then
\[\begin{aligned} N(G,H) &= N(G,H+x)+ N(G',H) =  N(G,H+x) + N(G_1,H) \\  S(G,H) &= S(G,H+x)+ S(G',H) =  S(G,H+x) + S(G_1,H). \end{aligned}\]
 By the induction hypothesis
\[ \begin{aligned}   S(G,H) &= S(G,H+x) + S(G_1,H)  \geq \frac{n+(h+1)}{2} N(G, H+x) + \frac{(n-1-m)+h}{2} N(G_1,H) \\&=
\frac{n+h}{2} ( N(G, H+x) +N(G_1,H)  ) + \frac{1}{2} \Big (N(G,H+x) - (m+1) N(G_1, H) \Big ) \\ & =
\frac{n+h}{2}  N(G, H) + \frac{1}{2} \Big (N(G,H+x) - (m+1) N(G_1, H) \Big )  \geq \frac{n+h}{2}  N(G, H).
\end{aligned}\]
The last inequality is proved as follows.  Denote the vertices in $V_2$ by $x_1,x_2, \dots, x_m$.  Let $x_0 = x$, and let $p_i, \, 0\leq i \leq m$, be a path from
a point in $H$ adjacent to $x$ to $x_i$.  These paths all contain $x$.  For each connected set $W$ counted by $N(G_1,H)$,  let $W_i, \,  0\leq i \leq m$, be the union of $W$ and the vertices of $p_i$.  
Therefore, for each connected set counted by $N(G_1,H)$, there are at least $m+1$ connected sets counted by  $N(G,H+x)$.
\end{proof}  

\begin{cor} \label{cor:v}
If $x$ is any vertex of a connected graph $G$ of order $n$, then
\[A(G,x) \geq \frac{n+1}{2}.\]
\end{cor} 

\begin{proof}  This is the case $h =1$ in Theorem~\ref{thm:m1}.
\end{proof}

\section{Near Trees} \label{sec:nt}

It will be helpful in investigating graphs with at least one cut-vertex to consider the {\bf block-cut tree} $\mathbb T = \mathbb T (G)$ of a graph $G$.  
The vertex set of $\mathbb T$ is the union of the cut-vertices of $G$ 
and the {\it blocks}, i.e., the maximal $2$-connected components, of $G$.  The latter include the edges of $G$.  
A cut-vertex $x$ and a block $B$ are adjacent in $\mathbb T$ if $x$ lies in $B$.  Call a block $B$ of $G$ a  {\bf leaf} if it is a leaf of
the tree $\mathbb T(G)$; otherwise call  $B$ {\bf interior}.  
Color the vertices $v$ in $\mathbb T$ corresponding
to cut-vertices in $G$ {\it blue} if $\deg_{\mathbb T}(v) \geq 3$. 
Color the corresponding vertices in $G$ also blue.  Color the vertices in $\mathbb T$ corresponding to a block $B$ in $G$ {\it red} if the
 order of $B$ is at least $3$.  
Color the corresponding blocks in $G$ also red.    

\begin{lemma} \label{lem:one3}  Assume that $G$ is a connected graph of order $n$ with exactly one red block $B$.  If $B$ has order $3$,
 then $A(G) > (n+2)/3$.
\end{lemma}

\begin{proof}  Let $v_1, v_2, v_3$ be the three vertices of $B$, and let $G_1, G_2, G_3$ be the corresponding 
connected components of $G\setminus \{v_1, v_2, v_3\}$.
For $i = 1,2,3$, let $G'_i$ be the tree induced by $V(G_i )\cup \{v_i\}$.  It is possible that $G_i$ is empty, in which case $G'_i = \{v_i\}$.  
Without loss of generality, let $G'_1, G'_2$ be the two with largest order.  Hence if $G'_1, G'_2$ have orders $n_1, n_2$, respectively,
then $n_1+n_2 \geq 2n/3$.  Let $e$ be the edge $\{v_1,v_2\}$ and let $T$ be the tree obtained by deleting $e$ from $G$.  From
\cite{J}, we know that $A(T) \geq (n+2)/3$.  If $\mathcal C  :=\mathcal C(G) \setminus \mathcal C(T)$, then let $N(\mathcal C) = |\mathcal C |$ and
let $S(\mathcal C)$ be the sum of the sizes of the sets in $\mathcal C$.  We claim that $S(\mathcal C)/ N(\mathcal C) > (n+2)/3$, which would
prove  Lemma~\ref{lem:one3}.  To simplify notation, let $N_i = N(G'_i,v_i)$ and $S_i = S(G'_i,v_i)$ for $i=1,2$.  Using Corollary~\ref{cor:v}, we have
\[\frac{S(\mathcal C)}{N(\mathcal C)} = \frac{S_1 N_2+S_2 N_1}{N_1 N_2} \geq \frac{(n_1+1)N_1 N_2 + (n_2+1)N_1 N_2 }{2 N_1 N_2} \geq 
\frac{2n/3 + 2}{2} = \frac{n+3}{3} >  \frac{n+2}{3}. \]
\end{proof}

\begin{figure}[htb] 
\vskip -4mm 
\includegraphics[width=6cm, keepaspectratio]{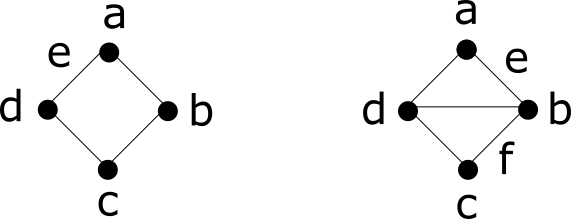} 
\vskip 2mm
\caption{Figure used in the proof of Lemma~\ref{lem:end}.}
\label{fig:0}
\end{figure}

\begin{lemma} \label{lem:end} Let $G$ be a graph that has no red interior blocks, and all red leaf blocks have order $3$ or $4$.  In 
addition, assume that all leaf blocks of order $4$ are of the form in Figure~\ref{fig:0} where vertex $a$ is the cut-vertex.  Then $A(G) > (n+2)/3$.
\end{lemma}

\begin{proof}  The proof is by induction on the number $m$ of red leaf blocks.  If $m=0$, then $G$ is a tree, and the 
result follows from \cite{J}.  Assume that the statement is true for $m-1$ and let $G$ be a graph with $m$ red leaf blocks.  

{\it Case 1.} Let $B$ be red leaf block of the type on the left in Figure~\ref{fig:0}.  If $H$ is the graph obtained from $G$ by deleting edge $e$, then $H$ has 
$m-1$ red blocks, and by the induction hypothssis $A(H) \geq (n+2)/3$ with equality if and only if $H$ is a path.   If $\mathcal C  :=\mathcal C(G) \setminus \mathcal C(H)$, then let 
$N(\mathcal C) = |\mathcal C |$ and let $S(\mathcal C)$ be the sum of the sizes of the sets in $\mathcal C$.  We claim that 
$S(\mathcal C)/ N(\mathcal C) > (n+2)/3$, which would prove Lemma~\ref{lem:end}.  Let $G'$ be the graph obtained from $G$ by deleting vertices $b,c,d$.  
To simplify notation, let $N = N(G',a)$ and $S = S(G', a)$.  Using Corollary~\ref{cor:v}, we have
\[  \frac{S(\mathcal C)}{ N(\mathcal C)} = \frac{(S+N)+(S+2N)+(S+2N)}{3 N} =\frac{S}{N} + \frac53 \geq \frac{(n-3)+1}{2} + \frac53 =
\frac{n}{2}+\frac{2}{3} > \frac{n+2}{3}.\]
In the first equality above, the term $(S+N)$ is for the connected sets containing just vertices $a$ and $d$; the first term $(S+2N)$ is for the connected sets containing just vertices $a,b$, and $d$; and the second $(S+2N)$ is for the connected sets containing just vertices $a, c$, and $d$.

{\it Case 2.} Let $B$ be red leaf block of the type on the right in Figure~\ref{fig:0}.  If $H$ is the graph obtained from $G$ by deleting edges $e.f$, then $H$ has 
$m-1$ red blocks, and by the induction hypothssis $A(H) > (n+2)/3$ (it cannot be a path).   If $\mathcal C  :=\mathcal C(G) \setminus \mathcal C(H)$, then let 
$N(\mathcal C) = |\mathcal C |$ and let $S(\mathcal C)$ be the sum of the sizes of the sets in $\mathcal C$.  We claim that 
$S(\mathcal C)/ N(\mathcal C) \geq (n+2)/3$, which would prove Lemma~\ref{lem:end}.  Let $G'$ be the graph obtained from $G$ by deleting vertices $b,c,d$.  
To simplify notation, let $N = N(G',a)$ and $S = S(G', a)$.  Using Corollary~\ref{cor:v}, we have
\[\frac{S(\mathcal C)}{N(\mathcal C)} = \frac{(S+N)+ 2}{N+1} \geq \frac{\big ((n-3)+1\big ) N/2  + N+2}{N+1} = \frac{nN+4}{2(N+1)} .\]
In the first equality above, the term $(S+N)$ is for the connected sets containing just vertices $a$ and $b$, and the term $2$ is for the single
connected set $\{b,c\}$.  It remains to show that  $(nN+4)(2N+2) \geq (n+2)/3$, which is equivalent to 
$(N-2)(n-4) \geq 0$, which holds.  

{\it Case 3.}  Let $B$ be a leaf block that is a $K_3$.  The proof in this case is a much simpler version of the proofs in cases 1 and 2.  
\end{proof}

\begin{definition}  \label{def:nt}
A {\bf near tree} is a graph $G$ such that one of the following holds:
\begin{enumerate}
\item $G$ is a tree;
\item $G$ has exactly one red block and that block is a $K_3$; or
\item $G$ has no interior red blocks and all leaf blocks have order $3$ or $4$.  
\end{enumerate}
\end{definition}

\begin{cor} \label{cor:nt}  If $G$ is a near tree, then $A(G) \geq (n+2)/3$, with equality if and only if $G$ is a path.  
\end{cor}

\begin{proof}  This is an immediate consequence of Lemmas~\ref{lem:one3} and \ref{lem:end}, and the fact from \cite{J}
that the statement is true for trees.  
\end{proof}

\section{An  inequality relating the number of connected sets containing a given vertex to the number of
connected sets not containing the vertex}  \label{sec:x}

Let $G$ be a connected graph and $x$ a vertex of $G$.  For ease of notation, let $G-x$ denote the subgraph of $G$ induced by $V(G) \setminus \{x\}$.
Let $T$ be a shortest distance spanning tree of $G$  rooted at $x$.  For each connected set $U \in \mathcal C(G-x)$ of vertices in $G-x$, choose a vertex $v_{U}$ that is closest to $x$, with distance being the length of the path $p_U$ in $T$ between $v_U$ and $x$.  Let $\overline U = U \cup p_U$, 
where we regard a path as its set of vertices. 
Let $\mathcal C(G,x)$ denote the set of connected sets in $G$ containing vertex $x$.  
For $Q \in \mathcal C(G,x)$, let
\[W(Q) = \{ U : U \in \mathcal C(G-x) \; \text{and} \;  \overline U = Q \}.\]
If $W(Q) \neq \emptyset$, 
there is a  linear order on $W(Q)$ defined by $U \preceq U'$ if $p_{U'} \subseteq p_U$.  Note that   $U \preceq U'$ implies that $U \subseteq U'$.
Let $U_Q$ denote the minimal set in $W(Q)$ with respect to this order, and let 
\[\mathcal M(G-x) = \{ U_Q : Q \in \mathcal C(G,x)\}\]
be the collection of all {\it minimals}. Note that $\{v\}$ is a minimal set for all $v\in V(G)$. Let  
\[av = av(G,x) =\frac{1}{|\mathcal M(G-x)|} \sum_{U \in \mathcal M(G-x)} |p_U|\] be the average length of the paths $p_U$ over all minimals $U$. 
Here  $|p_U|$ denotes the length of path $p_U$, i.e., the number of edges.

\begin{theorem} \label{thm:av}  For a connected graph $G$ and vertex $x$, we have
 \[ av(G,x) \cdot (N(G,x) - 1)  \geq N(G-x).\]
\end{theorem}

\begin{proof}
For a minimal set $U \in \mathcal M(G-x)$,  let $p_U = \{v_U = p_0, p_1, p_2, \dots, p_{k} = x\}$, vetices of $p_U$ in succession, 
where $k := k_U$ depends on $U$. Let
\[ Y(U) = \{ U \cup \{p_0, p_1, p_2, \dots, p_j \}: 0\leq  j < k_U\}.\]   
Clearly $|Y(U)| = |p_U|$.  Note  that the sets $Y(U)$ are pairwise disjoint, i.e., if $U\neq U'$, then $Y(U)\neq Y(U')$, and
\[\mathcal C (G-x) = \bigcup_{U \in \mathcal M(G-x)} Y(U).\] 
Thus $\{Y(U) : U \in \mathcal M(G-x) \}$ partitions $\mathcal C(G-x)$. 
 Consider the map $f: \mathcal C (G-x) \rightarrow \mathcal C (G,x)$ defined by $f(U) =\overline U$.  For $U \in \mathcal M(G-x)$, each set in $Y(U)$ is mapped
 to the same set in $\mathcal C (G,x)$; for distinct $U,U' \in \mathcal M(G-x)$ each pair of sets $A  \in Y(U)$ and $B \in Y(U')$ are mapped to 
distinct sets in $\mathcal C (G,x)$.  Therefore
\[N(G-x) = \sum_{U \in \mathcal M(G-x)} |p_U|  = av(G,x) \cdot |\mathcal M(G-x)| \qquad \text{and}  \qquad N(G,x) \geq |\mathcal M(G-x) |+1.\]
 The $+1$ is to count the vertex $x$ itself.  
Therefore $N(G,x) \geq N(G-x)/av(G,x) + 1$ and hence $av(G,x) \cdot (N(G,x) - 1)  \geq N(G-x)$.  
\end{proof}

\begin{theorem} \label{thm:2av} For any $2$-connected graph $G$ of order $n$ and any vertex $x$ of $G$, we have $av(G,x) \leq (n-1)/2$,
with equality if and only if $G=K_3$.  
\end{theorem}

\begin{proof}  According to \cite[Theorem 1]{C}, the diameter of a $2$-connected graph is at most $\lceil (n-1)/2\rceil$.  If $n$ is odd, then the
diameter is at most $(n-1)/2$.  Then clearly  average $av(G,x) < (n-1)/2$ unless $G = K_3$, in which case $av(G,x) = 1 = (n-1)/2$. 

If $n$ is even, let $y$ be a vertex of $G$ furthest from $x$. If the the distance from $x$ to $y$ is less than $n/2$, then 
it is clear that  $av(G,x)  \leq n/2 -1 < (n-1)/2$.  So assume that the distance between $x$ and $y$ is exactly $n/2$.  Since $G$ is $2$-connected, there is a cycle $C$ containing $x$ and $y$.    Because the distance between $x$ and $y$ is exactly $n/2$, the cycle $C$ contains all the vertices of $G$. A minimum distance spanning tree $T$ of $G$ contains all edges of $C$ except one that is incident to $y$, say $\{y,w\}$.  In this case
 $\mathcal M(G-x) = \{ \{v\} : v \in V(G), v\neq y\}\cup\{ \{w,y\} \}$. 
A simple calculation shows that $av(G.x) = n/4 + (1/2 - 1/n) < (n-1)/2$.  
\end{proof}

\begin{cor} \label{cor:av1} Let $H$ be a maximal $2$-connected subgraph of order at least $3$ of a graph $G$ of order $n$.   
Let $x$ be a vertex of $H$
 such that  the set of all neighbors of $x$ induce a complete subgraph of $H$.  Then $av(G,x) \leq (n-1)/2$, 
  with equality if and only if $G=K_3$. 
\end{cor}

\begin{proof}  Theorem ~\ref{thm:2av} settles the case $G=H$, so assume that $H$ is a proper subgraph of $G$.  Construct a shortest distance 
spanning tree $T$ of $G$ by first constructing a shortest distance  spanning tree $T$ of $H$ and extending it to $G$.  Denote the order of $H$ by $h$.
Let $K$ denote the complete graph induced by the neighbors of $x$.  Note that no edge of $K$ is in $T$. 

Consider a set $U \in \mathcal M(G-x)$ such that $U$ has a vertex in $H$. As in the proof of Theorem~\ref{thm:2av}, if $h$ is odd, then
$|p_U| \leq (n-1)/2$.  If $h$ is even, then either $|p_U| \leq n/2 - 1$ or $U\cap H = \{y\}$, where $y$ is the unique vertex of $H$ at distance
$h/2$ from $x$.  We claim that $U\cap H = \{y\}$ is not possible.  Note that all edges incident with $y$ that are not in $H$ must be in $T$.  
Therefore, if $U' := U \setminus \{y\}$, then $U'  \prec U$, which implies that $U\notin \mathcal M(G-x)$, a contradiction.  

Next partition $\mathcal M(G)$ into three sets $A, B, C$ as follows.  We will consider the average 
of the $|p_U|$ for $U$ in each of the sets $A, B, C$.
Let $A$ consist of all those connected sets $U\in \mathcal M(G)$ such that $U$ has a vertex in $H$
but does not contain all vertices in $K$. 
Let $B'$ consist of all those connected sets $U \in \mathcal M(G-x)$ with no vertex in $H$.  For $U \in B'$ let $p$ be the subpath of 
$p_U$ with one end vertex in $K$ and the other in $U$.  
Let $U' = U \cup p \cup K$, and note that $U' \in \mathcal M(G)$ and $|p_{U'}| = 1$.  The map $f : B' \rightarrow \mathcal M(G)$ defined by 
$f(U) = U'$ is an injection. Let $B = B' \cup f(B')$.  Let $C$ be the complement of $A\cup B$ in $\mathcal M(G)$, and note
that  $|p_{U}| = 1$ for all $U \in C$.  

We have aleady shown that the average of the path distances $|p_U|$ for $U\in A$ is at most $(n-1)/2$ if $h$ is odd and at most ess than $(n-2)/2$
if $h$ is even.  Concerning the set $B$, the average 
\[(|p(U)|+|p(U')|)/2 \leq \begin{cases} \frac12 \big (( n-h) +\frac{h}2+1  \big ) = \frac12 \big (n- \frac{h}{2} + 1  \big ) \leq \frac{n-1}2
 \quad \text{if} \; h \; \text {is even} \\
\frac12 \big ( (n-h)+ \frac{h-1}{2} + 1  \big )  = \frac12 \big (n- \frac{h}{2} + \frac12  \big )  < \frac{n-1}2\quad \text{if} \; h \; \text{is odd}. \end{cases} \]
Therefore,  the average of the path distances $|p_U|$ for $U\in B$ is at most $(n-1)/2$. 
The path distance $|p_U| = 1$ for all connected sets  in $C$.  Therefore we have $av(G,x) < (n-1)/2$ unless $G = K_3$.  
\end{proof}

\section{An Inequality For Graphs With A Cut-Vertex} \label{sec:cut}

Let $x$ be a cut-vertex of a connected graph $G$ of order $n$, and let $M=M(x) = M(G,x)$ denote the number of  connected components of
$G-x$.  Denote these components by $G_1, \dots, G_M$, and let $n_1, \dots , n_M$ be their respective orders. 
Note that $n = 1+ n_1+n_2+\cdots + n_M$.   For $i=1,2, \dots, M$, denote by $G'_i $ the subgraph of $G$ induced
by the vertices $V(G_i)\cup \{x\}$.  To simplify notation, let $N_i= N(G_i)$ and $N_i(x) = N(G'_i,x)$.  Let $a_i = av(G'_i,x)/n_i$.

The main result of this
section is the following inequality, which is essential to our proof of Theorem~\ref{thm:main}. The proof of Theorem~\ref{thm:inequal} appears at the
end of this section, after several lemmas. 

\begin{theorem} \label{thm:inequal}  If $G$ is a connected graph with at least one cut-vertex, but not a near tree, then there is a cut-vertex $x$ such that
following inequality holds:
\begin{equation} \label{eq:main}  (n-1)  \prod_{i=1}^M N_i(x)  > 2 \sum_{i=1}^M  (n-n_i ) N_i. \end{equation}
\end{theorem}

The cut-vertex $x$ in Therorem~\ref{thm:inequal} will be called the {\bf root vertex} of $G$.  The theorem states that we can choose
a root vertex that satisfies inequality~\eqref{eq:main}.

\begin{lemma} \label{lem:twice}   If $x$ is a vertex of degree at least $2$ in a tree $T$ of order $n$, 
then $N(T,x) \geq 2n$ unless $T$ is one of the trees in Figure~\ref{fig:1}.
\end{lemma}

\begin{proof}   It is routine to check that if $T$ is a tree of order at most $6$ with $\deg(x) \geq 2$, not a tree in Figure~\ref{fig:1}, then $N(T,x) \geq 2n$.  
Proceeding by induction on $n$, assume the statement is true for all trees of order $n$ with $n\geq 6$ and not in Figure~\ref{fig:1}, and let $T$ be tree of order $n+1$.
Remove a leaf $y$ of $T$, not a child of $x$, to obtain a tree $T'$ of order $n$.  By the induction hypothesis, either $N(T',x) \geq 2n$ or $T'$ is
a graph of the form on the left in Figure~\ref{fig:1}.  In the first case, adding $y$ back adds at least two new connected subtrees containing $x$,
the path $p$ from $x$ to $y$ and the union of $p$ and a child of $x$ not on $p$. Therefore $N(T,x) \geq N(T',x) \geq 2(n+1)$.
 In the second case, $T$ is of the type  on the left in Figure~\ref{fig:1}.  
\end{proof}

\begin{figure}[htb] 
\vskip -6cm
\hskip -7cm
\includegraphics[width=16cm, keepaspectratio]{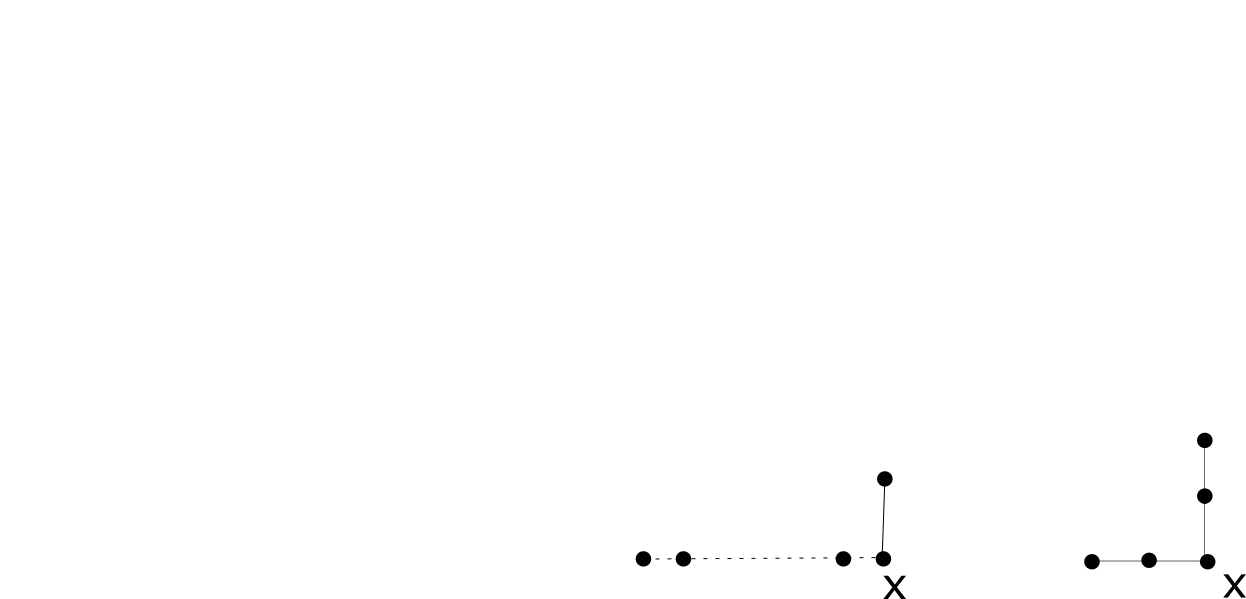} 
\vskip 3mm
\caption{A dashed line indicates any number of vertices. }
\label{fig:1}
\end{figure}

\begin{cor} \label{cor:twice1}   If $x$ is a vertex in a graph $G$ of order $n$, and $x$ is contained in a $2$-connected subgraph of $G$,
then $N(G,x) \geq 2n$ unless $G$ is a subgraph containing $x$ of
one of the graphs in Figure~\ref{fig:2}.  
\end{cor}

\begin{proof}  Let $T$ be a spanning tree of $G$ containing all edges incident to $x$.  By Lemma~\ref{lem:twice} we have  $N(G,x) \geq N(T,x) \geq 2n$  
unless $T$ is one 
of the trees in Figure~\ref{fig:0}.   In this latter case, however, the fact that $x$ lies in a $2$ connected graph implies  the existence of sufficiently many
additional edges to insure that $N(G,x) \geq 2 n$. 
\end{proof}

\begin{figure}[htb] 
\vskip -1mm
\centering
\includegraphics[width=11cm, keepaspectratio]{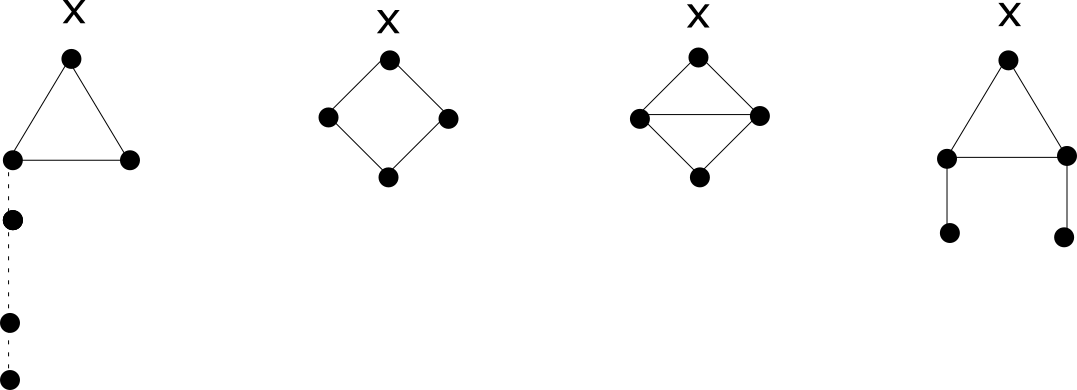} 
\vskip -1cm
 \caption{The dashed line indicates any number of vertices.}
\label{fig:2}
\end{figure}
\vskip 3mm

If $M(x)=2$ in the statement of Theorem~\ref{thm:main}, then inequality~\eqref{eq:main}  reduces to
\begin{equation} \label{eq:main2}  (n-1)  N_1(x)N_2(x)  > 2 (n_2+1) N_1 + 2 (n_1+1) N_2. \end{equation}

\begin{lemma} \label{lem:d2}   If $x$ is a cut-vertex of $G$ such that $M(x) = 2$ and both 
$N(G'_i,x) \geq 2(n_i+1)$ and $a_i \leq 1/2$ for either $i=1$or $i=2$,
 then  
\[ (n-1)  N_1(x)N_2(x)  > 2 (n_2+1) N_1 + 2 (n_1+1) N_2. \]\end{lemma}

\begin{proof}  
Without loss of generality, assume that
$N_1(x) \geq 2(n_1+1)$ and $a_1 \leq 1/2$.
Then using Theorem~\ref{thm:av} and the obvious fact that $N_2(x) \geq n_2+1$
we have
\[\begin{aligned}   (n-1) & N_1(x)N_2(x)  - \big (2(n_2+1) N_1 + 2 (n_1+1) N_2 \big ) \\ &= \big (  n_1  N_1(x)N_2(x) - 
2 (n_2+1) N_1 \big ) + \big (  n_2  N_1(x)N_2(x) -  2 (n_1+1) N_2 \big ) \\ & \geq \big (  n_1  N_1(x)N_2(x) - 
2 (n_2+1) a_1 n_1 (N_1(x) -1) \big ) + \big ( n_2  N_1(x)N_2(x) -  2(n_1+1)  n_2 N_2(x) \big ) \\ & > 
 \big (  n_1  N_1(x)N_2(x) - 
2 (n_2+1) a_1 n_1 N_1(x) \big ) + \big ( n_2  N_1(x)N_2(x) -  2(n_1+1)  n_2 N_2(x) \big ) \\&=
n_1 N_1(x)\big ( N_2(x)-2(n_2+1) a_1\big ) + n_2 N_2(x)\big ( N_1(x)-2(n_1+1)\big ) \\& \geq
n_1 N_1(x)\big ( N_2(x)-(n_2+1)\big ) + n_2 N_2(x)\big ( N_1(x)-2(n_1+1)\big ) 
 \geq 0.  \end{aligned} \]
\end{proof}

\begin{lemma} \label{lem:cut1}   If there is a cut-vertex $x$ in $G$ such that either 
\begin{enumerate}
\item  $M(x) \geq 4$ or  
\item  $M(x)=3$ with $\min \{n_1,n_2, n_3\} \geq 2$ and at least two of $n_1,n_2, n_3$ are at least $3$,
\end{enumerate}
then inequality~\eqref{eq:main} holds with $x$ as the root of $G$.
\end{lemma}

\begin{proof}  Using Theorem~\ref{thm:av} and the fact that $N_i(x) \geq n_i+1$ we have
\[\begin{aligned} (n-1) & \prod_{i=1}^M N_i(x)  - 2\sum_{i=1}^M  (n-n_i ) N_i = \sum_{i=1}^M \Big ( n_i N_1(x) N_2(x) \cdots N_M(x) - 
2(n-n_i ) N_i  \Big ) \\ & >  \sum_{i=1}^M \Big ( n_i N_1(x) N_2(x) \cdots N_M(x) - 2(n-n_i ) a_i  n_i N_i (x)\Big ) \\ & \geq  
\sum_{i=1}^M  n_i N_i(x) \Big  ( \prod_{j\neq i}  N_j(x) - 2 \big (1+ \sum_{j\neq i} n_j \big )\Big )
 \geq 
\sum_{i=1}^M  n_i N_i(x) \Big  (\prod_{j\neq i} (n_i +1) - 2 \big (1+ \sum_{j\neq i} n_j \big )\Big ).
\end{aligned} \]
If $M\geq 4$, then 
\begin{equation} \label{eq:ip} \prod_{j\neq i} (n_i +1) \geq 2 \big (1+ \sum_{j\neq i} n_j \big ),\end{equation}
verifying  inequality~\eqref{eq:main}.
If $M=3$, then, without loss of generality, assume that $i = 3$ in inequality~\eqref{eq:ip}, in which case
\[\prod_{j\neq i} (n_i +1) - 2 \big (1+ \sum_{j\neq i} n_j \big ) = n_1 n_2 - n_1 - n_2 -2 = (n_1-1)(n_2-1) -2,\]
which is greater than or equal to $0$ if $\min \{n_1,n_2\} \geq 2$ and $\max \{n_1,n_2\} \geq 3$.  Therefore
 \[(n-1)  \prod_{i=1}^M N_i(x)  > 2 \sum_{i=1}^M  (n-n_i ) N_i\]
if $M\geq 4$ or if $M=3$ with $\min \{n_1,n_2, n_3\} \geq 2$ and at least two of $n_1,n_2, n_3$ at least $3$.  
\end{proof}

\begin{lemma}  \label{lem:cut2}
 Let $x$ be a cut-vertex of $G$ with $M(x) \geq 3$ and denote the components of $G-x$ by $G_1, G_2, \dots, G_M$. 
Assume, without loss of generality, that $ \min\{ N_i : 1\leq i \leq M\} = N_M$.  
If the  graph $G'$ induced by the vertices $\{x\} \cup \bigcup_{i=1}^{M-1} V(G_i)$ satisfies \eqref{eq:main}, then $G$ also  satisfies \eqref{eq:main}.
\end{lemma}

\begin{proof}  Assume that $(n-1)  \prod_{i=1}^{M-1} N_i(x)  > 2 \sum_{i=1}^{M-1}  (n-n_i ) N_i$ holds for the graph $G'$.  We must show
that 
\[ (n+n_M-1 )\prod_{i=1}^{M} N_i(x) \\  > 2 \sum_{i=1}^{M-1}  (n+n_M-n_i ) N_i + 2 n N_{M},\] i.e.,
\[ N_M(x) (n-1) \prod_{i=1}^{M-1} N_i(x)  +n_M  N_M(x) \prod_{i=1}^{M-1} N_i(x)  >  2 \sum_{i=1}^{M-1}  (n-n_i ) N_i  + 2 n_M \sum_{i=1}^{M-1} N_i + 2 n N_{M}.\] 
Because it is assumed that  $G'$  satisfies \eqref{eq:main2}, this reduces to showing that
\[ N_M(x)  2 \sum_{i=1}^{M-1}  (n-n_i ) N_i  +n_M  N_M(x) \frac{2 \sum_{i=1}^{M-1}  (n-n_i ) N_i }{n-1} \geq  2 \sum_{i=1}^{M-1}  (n-n_i ) N_i  + 2 n_M \sum_{i=1}^{M-1} N_i + 2 n N_{M},\] 
i.e.,
\[\sum_{i=1}^{M-1} \Bigg ( (N_M(x)-1) (n-n_i )  +n_M \Big ( \frac{  N_M(x)(n-n_i)}{n-1} -1 \Big ) \Bigg) N_i \geq n N_{M}.\]
By the minimality of $N_M$ and the fact that $\sum_{i=1}^{M-1} n_i = n-1$, it now suffices to show that

\[\begin{aligned}   \Big ( N_M(x)-1 &+\frac{ n_M N_M(x)}{n-1}\Big ) \big  (M n-2n+1 \big )  - (M-1) n_M\\ &=
\Big ( N_M(x)-1 +\frac{ n_M N_M(x)}{n-1}\Big ) \sum_{i=1}^{M-1} (n-n_i) - (M-1) n_M \\ &=
\sum_{i=1}^{M-1} \Bigg ( (N_M(x)-1) (n-n_i )  +n_M \Big ( \frac{  N_M(x)(n-n_i)}{n-1} -1 \Big ) \Bigg) \geq n.
\end{aligned} \]
Because $N_M(x) \geq n_M+1$, we have
\[ N_M(x)-1 +\frac{ n_M N_M(x)}{n-1} \geq \frac{n_M(n_M+n)}{n-1}.\]
To finish the proof, the following inequality is required:
\[ (M n-2n+1 \big )\big (n_M (n_M+n) \big ) -  \big ((M-1) n_M + n\big )(n-1) \geq 0.\]
As a function of $M$, the derivative of the left hand side of the inequality above is positive. Therefore it is sufficient to prove the inequality for $M=3$, i.e.,
\[(n+1)(n_3^2+n n_3) - (n-1)(2 n_3 + n) = n^2(n_3-1) + n n_3(n_3+1-2) +n_3^2+2n_3+n \geq 0,\]
which clearly holds.
 \end{proof}

\begin{proof}[{\bf Proof of Theorem~\ref{thm:inequal}}]  

If $G$ has a cut-vertex that satisfies condition (1) or  (2) in the hypothesis of Lemma~\ref{lem:cut1},  then, by that lemma, Theoerem~\ref{thm:inequal} is true.  Therefore it can be assumed that, for all cut-vertices $x$ of $G$,  either
\begin{enumerate}
\item[(a)] $M(x) = 2$, or
\item[(b)] $M(x) = 3$ and at least two components of $G-x$ have order at most $2$.
\end{enumerate}
In case (b), if $x$ is chosen as the root of $G$, then by Lemma~\ref{lem:cut2} it may be assumed that
\begin{enumerate}
\item[(b$'$)] a component of $G-x$ of smallest order has been removed and $M(x) = 2$ for the resulting graph.  
\end{enumerate}
\vskip 2mm

The proof is by cases.  
We will show that the inequality~\eqref{eq:main} holds when $G$ has:
\begin{enumerate}
\item  a red block of order at least $5$;
\item  an interior red block of order $4$;
\item  at least two red blocks, one of which is interior.
\end{enumerate}
Our assumption that $G$ is not a near tree eliminates the cases:
\begin{itemize}
\item $G$ is a tree;
\item $G$ has exactly one interior block of order $3$; 
\item all leaf blocks are of order $3$ or $4$.  
\end{itemize}
Hence cases 1-3 are exhaustive and their proof is sufficient to vertify Theorem~\ref{thm:inequal}.
\vskip 2mm

{\it Case 1.}  Let $B$ be a block in $G$ of order at least $5$, and let $x$ a cut-vertex in $B$.  
Choose $x$ as the root of $G$.  By items $(a)$ and $(b')$ above, it may be assumed that $M(x) = 2$.  
Inequality \eqref{eq:main2} must be verified.
  
Let $G'$ be the graph induced by the union of $x$ and the component of $G-x$ containing $B$.  
 Without loss of generality, let this be the component whose parameters have index $i=1$ in inequality \eqref{eq:main2}.  
Let $\widehat G$ be obtained from $G'$ by adding an edge between every pair of neighbors of $x$.  
Note that $G'$ and $\widehat G$ have the same set of vertices and the same set of connected sets containing $x$.  Also,
the number of connected sets in $\widehat G$ not containing $x$ is at least as large as the number of connected sets in $G'$ not containing $x$. 
Therefore if inequality \eqref{eq:main2} holds with $G'$ replaced by $\widehat G$, then it also holds for $G$. 
By Corollary~\ref{cor:av1} we have $a_1 \leq 1/2$, and by Corollary~\ref{cor:twice1} we have 
$N(\widehat G,x) \geq 2(n_1+1)$ - unless $\widehat G$ is a subgraph containing $x$ of one of the graphs in Figure~\ref{fig:2}.  
This is not possible since the order of $B$ is at least $5$.   By Lemma~\ref{lem:d2}, the proof of Case 1 is complete.  
\vskip 2mm

{\it Case 2.}   Let block $B$ in $G$ be of order $4$ and, since $B$ is interior, let $x$ and $y$ be distinct cut-vertices of $G$ on $B$.
Take $x$ as the root of $G$.  Let $G'$ be the graph induced by 
the union of $x$ and the component of $G-x$ containing $B$.  Note that $G'$ contains $y$, and therefore $G'$
cannot be a subgraph containing $x$ of any
graph in Figure~\ref{fig:2}. Hence $N(G',x) \geq 2(n_1+1)$ by Corollary~\ref{cor:twice1}.  
Now the proof proceeds as in Case 1.
\vskip 2mm

{\it Case 3.}    Let $B_1$ be a red block that is interior; let $B_2$ be another red block; and let $p$ be the unique path in $\mathbb T(G)$ joining $B_1$ and $B_2$.  Since $B_1$ is interior, by Case 2 we can assume that it has order $3$. 
Let $x'$ be a vertex in $\mathbb T(G)$ adjacent to $B_1$ that does not lie on $p$.   Let $x$ be the vertex in $G$ corresponding to $x'$, 
and choose $x$ as the root of $G$.  We may assume by $(a)$ and $(b')$ above that $M(x)=2$, which reduces the problem to proving inequality \eqref{eq:main2}.  Let $G'$ be the graph induced by the union of $x$ and the component of $G-x$ containing $B_1$ and $B_2$. 
By Corollaries~\ref{cor:av1} and \ref{cor:twice1}, we have $N(\widehat G,x) \geq 2(n_1+1)$ and $a_1 \leq 1/2$.
Lemma~\ref{lem:d2} completes the proof of Case 3.  
\end{proof}

\section{Proof of the Lower Bound Theorem} \label{sec:lb}

\begin{prop} \label{prop:nx}  If $G$ is a connected graph with vertex set $\{x_1, \dots, x_n\}$, then $S(G) = \sum_{i=1}^n N(G,x_i)$.
\end{prop}

\begin{proof} Count the number of pairs $(x,U)$ such that  $x\in V(G), \; U \in \mathcal C (G)$ and $x\in U,$ in two ways to obtain
\[ S(G) = \sum_{U \in \mathcal C}  |U| = \sum_{x \in V(G)} N(G,x) = \sum_{i=1}^n N(G,x_i).\]
\end{proof}

\begin{theorem} \label{thm:bound}  For a connected graph $G$ of order $n$ we have
\[ S(G) \geq \frac{n+2}{3}\, N(G),\]
with equality if and only if $G$ is a path.  
\end{theorem}

\begin{proof}  The proof is by induction on $n$.  The statement is easily checked for $n \leq 4$.  By Corollary~\ref{cor:nt}, it is also true for 
near trees as in Definition~\ref{def:nt}.  Assume it is true for graphs of order $n-1$, and let $G$ have order $n$.  
By Proposition~\ref{prop:nx}, the average of the numbers $N(G,x)$ over all vertices $x$ in $G$ is $S/n$.  
Let $x$ be a vertex such that $N(G,x) \geq S(G)/n$. 
Let $G' = G-x$. There are two cases.
\vskip 1mm

{\it Case 1.}  The vertex $x$ is not a cut-vertex, hence $G-x$ is connected.  From Theorem~\ref{thm:m1} and the induction hypothesis
\[\begin{aligned} S(G) &= S(G-x)+ S(G,x) \geq \frac{n+1}{3} N(G-x) + \frac{n+1}{2}  N(G,x) \\ &=  
\frac{n+2}{3} \big ( N(G-x) + N(G,x)\big ) - \frac13 N(G-x) + \frac{n-1}{6} N(G,x) \\
&= \frac{n+2}{3}  N(G) +\frac16 \Big ( (n-1) N(G,x) - 2 N(G-x) \Big ) \\ & \geq 
 \frac{n+2}{3}  N(G) +\frac16\Big ( (n-1) S(G)/n - 2 N(G-x) \Big ) 
\end{aligned}\]
It only remains to show that $(n-1) S(G)/n > 2 N(G-x)$.   But we have
\[ (n-1) S(G) = (n-1)\Big (   S(G-x)+ S(G,x) \ \Big )  \geq (n-1) \Big (\frac{n+1}{3} N(G-x) + \frac{n+1}{2}  N(G,x)   \Big ),\]
which is larger than $2 n N(G-x)$ if and only if
\[ 3(n^2-1) N(G,x) + 2(n^2-6n-1) N(G-x) > 0.\]
The polynomial $2(n^2-6n-1)$ is positive for $n\geq 7$.  The inequality for smaller values of $n$ can be checked using the facts that, for $n=4$ we have
$N(G,x) \geq 4, \, N(G-x)\leq 7$ (the path and the complerte graphs), for $n=5$ we have $N(G,x)\geq 5, \, N(G-x) \leq 15$, and for $n=6$ we have 
$N(G,x)\geq 6, \, N(G-x) \leq 31$.
\vskip 1mm

{\it Case 2.}  Every vertex such that $N(G,x) \geq S(G)/n$ is a cut-vertex. If $G$ is a near tree, then we are done.
Otherwise, by Theorem~\ref{thm:inequal}, there is cut-vertex $x$ that satisfies 
inequality~\eqref{eq:main}.  Let $G_1, \dots, G_M$ be the connected components of
$G-x$. Note that $n = 1+ n_1+n_2+\cdots + n_M$.   Now
\[\begin{aligned} N(G,x) &= \prod_{i=1}^M N(G'_i,x) \\  S(G,x) &= \sum_{i=1}^M (S(G'_i,x) -  N(G'_i,x) ) \prod_{j \neq i} N(G'_j,x) + \prod_{i=1}^M N(G'_i,x) \\
&=\sum_{i=1}^M S(G'_i,x) \prod_{j \neq i} N(G'_j,x) - (M-1) \prod_{i=1}^M N(G'_i,x).
\end{aligned}\] 
In the formula for $S(G,x)$, the terms  $-  N(G'_i,x) )$ and   $\prod_{i=1}^M N(G'_i,x)$ are to count the vertex $x$ the correct number of times.  
By the induction hypothesis and Theorem~\ref{thm:m1} we have
\[\begin{aligned}  S(G) &= S(G-x) + S(G,x) = \sum_{i=1}^M S(G_i) + \sum_{i=1}^M S(G'_i,x) \prod_{j \neq i} N(G'_j,x) - (M-1) \prod_{i=1}^M N(G'_i,x) \\ & \geq
\sum_{i=1}^M \frac{n_i+2}{3} N(G_i) + \sum_{i=1}^M \Big ( \frac{n_i+2}{2} N(G'_i,x)  \prod_{j \neq i} N(G'_j,x) \Big )- (M-1) \prod_{i=1}^M N(G'_i,x)  \\ & =
\sum_{i=1}^M \frac{n_i+2}{3} N(G_i) + \frac{n+1}{2} \prod_{i=1}^M N(G'_i,x).
\end{aligned}\]
It remains to show that the expression in the last line above is greater than
\[\frac{n+2}{3} N(G) = \frac{n+2}{3} \big ( N(G-x) + N(G,x)\big ) =\frac{n+2}{3} \Big ( \sum_{i=1}^M  N(G_i) +   \prod_{i=1}^M N(G'_i,x) \Big ).\]  
This is equivalent to showing that 
\[ (n-1)  \prod_{i=1}^M N(G'_i,x)  > 2 \sum_{i=1}^M  (n-n_i ) N(G_i), \]
which is exactly inequality~\eqref{eq:main} in Theorem~\ref{thm:inequal}.  
\end{proof}

\section{Two Open Problems} \label{sec:open}  

Although, for a general connected graph, the lower bound of Theorem~\ref{thm:main} is best possible, evidence indicates that 
$D(G) > 1/2$ for a large class of graphs.  The
result of Kroeker, Mol, and Oellermann \cite{O} referenced in Section~\ref{sec:lit}, for example, proves that this is the case for cographs.  
We made the following conjectured in \cite{V3}. 
 
\begin{conjecture}  For any graph $G$, all of whose vertices have degree at least $3$, we have $D(G) > \frac12$.  
\end{conjecture}

One difficulty in proving this conjecture, if true, is that knowing exactly for which graphs $D(G) > \frac12$ is problematic .  There are
graphs, all of whose vertices have degree at least $2$, whose density is less than $1/2$ and some whose density is greater.  Adding an
edge to a graph may increase the density or it may decrease the density, similarly for adding a vertex. This makes a proof by
induction challenging.  

As mentioned in Section~\ref{sec:lit}, there are trees whose density is arbitrarily close to $1$.  For trees where every vertex has 
degree at least $3$, the density is bounded above by $3/4$ and this is best possible \cite{V}. A family of cubic graphs appearing in \cite{V3} has
asymptotic density $5/6$.  We know of no graph, all of whose vertices have degree at least $3$, with a larger density.    

\begin{question}  Is there an upper bound, less than $1$, on the density of graphs all of whose vertices have degree at least $3$?
\end{question}

\end{document}